\documentclass[11pt]{article}
\usepackage{amsfonts,color}
\usepackage{amssymb}
\usepackage{amsmath}
\usepackage{amsthm}
\usepackage{enumerate}
\usepackage{graphicx}
\usepackage{wrapfig}
\usepackage{rotating}
\usepackage{epsfig}
\usepackage[left=1in, right=1in, top=1in, bottom=1in]{geometry}
\usepackage{mdwlist}
\usepackage{secdot}
\usepackage{comment}
\usepackage{mathrsfs}
\usepackage{mdwlist}
\usepackage{multirow}
\usepackage{lscape}
\usepackage{pdflscape}
\usepackage{array}
\usepackage[compact]{titlesec} 
\usepackage{times}
\usepackage{tikz}
\usepackage{tikzscale}
\usepackage[justification=centering, labelfont=bf]{caption}
\usepackage{setspace}
\usepackage{color}
\usepackage{authblk}
\usepackage{wrapfig}
\usepackage{hyperref}
\usepackage[title]{appendix}
\usepackage[ruled, linesnumbered]{algorithm2e}
\usepackage{float}
\usepackage{optidef}
\usepackage{url}
\usepackage{subcaption}
\usepackage[title]{appendix}
\usepackage{etoolbox}
\usepackage{units}
\usepackage{bm}
\allowdisplaybreaks

\newtheorem{theorem}{Theorem}
\newtheorem{lemma}[theorem]{Lemma}

\newtheorem{proposition}[theorem]{Proposition}
\newtheorem{corollary}[theorem]{Corollary}

\DeclareMathOperator*{\R}{\mathbb{R}}
\DeclareMathOperator*{\E}{\mathbb{E}}
\DeclareMathOperator*{\N}{\mathbb{N}}
\DeclareMathOperator*{\Z}{\mathbb{Z}}
\DeclareMathOperator*{\Prob}{\mathbb{P}}

\newcommand{\bigO}{\mathcal{O}}

\newcommand{\dv}{\text{d}}

\newcommand{\FBS}{\textsf{FBS}}
\newcommand{\BR}{\textsf{BR}}
\newcommand{\IRT}{\textsf{IRT}}
\newcommand{\LT}{\textsf{LT}}

\newcommand{\revision}[1]{{\color{black}#1}}

\usepackage[comma]{natbib}
\setcitestyle{authoryear,open={(},close={)},aysep={}}

\begin{document}

\title{\vspace{-40pt}On Linear Threshold Policies for\\[-12pt]Continuous-Time Dynamic Yield Management}
\author{
Dipayan Banerjee$^{1}$
\qquad \quad 
Alan Erera$^{2}$
\qquad \quad Alejandro Toriello$^{2}$
\vspace{-10pt}}
\affil{\small{$^1$Quinlan School of Business, Loyola University Chicago}\\
\small{$^2$H.\ Milton Stewart School of Industrial and Systems Engineering, Georgia Institute of Technology}}

\date{}

\maketitle
\vspace{-20pt}


\noindent \textbf{Abstract.} 
We study the finite-horizon continuous-time dynamic yield management problem with stationary arrival rates and two customer types. We consider a class of linear threshold policies proposed by \cite{Hodge2008}, in which each less-profitable customer is accepted if and only if the remaining inventory exceeds a threshold that linearly decreases over the horizon. We use a Markov chain representation to show that such policies achieve uniformly bounded regret. We then generalize this result to analogous policies for arbitrarily many customer types.

\section{Background}
\label{sec:intro}

A multichannel retailer sells a product through two different channels, denoted as \textit{offline} and \textit{online} channels for consistency, over a finite horizon of duration $T$. Customers arrive over the horizon and attempt to purchase one unit of the product. Offline customers arrive via a Poisson process with rate $\lambda_1 > 0$ and pay $p_1 > 0$ for each unit. Online customers arrive via a Poisson process with rate $\lambda_2 > 0$ and pay the same price. The retailer incurs an additional fixed cost to fulfill each online order, and online customers receive free delivery; this reflects the situation faced by an industry partner. Equivalently, we instead assume without loss of generality that online customers pay $p_2 \in (0,  p_1)$ for each unit, and the retailer does not incur the fixed per-customer cost.
Any inventory left over at the end of the selling horizon is discarded and provides no additional salvage value. For notational consistency with related work, $t$ denotes the time left until the end of the horizon; thus, $t = 0$ denotes the terminal time. The retailer has $n$ units of starting inventory available at $t = T$ with no replenishment opportunities during the horizon. 

The objective is to maximize the revenue collected over the horizon. 
To this end, the retailer must immediately decide whether to accept an arriving customer's order. It is clear that every offline customer should be accepted (thereby receiving $p_1$ in revenue) while positive inventory remains. Backorders are not permitted, so all customers are rejected upon inventory depletion. Thus, the retailer faces the problem of determining whether to accept or reject each arriving online customer (thereby receiving $p_2$ or $0$ in revenue, respectively). Mathematically, this is the \textit{continuous-time dynamic yield management problem} with two customer classes, stationary arrival rates, and a finite horizon. 

As shown independently by \cite{Feng2001, Liang1999, Zhao1999}, a particular threshold policy maximizes expected revenue. Specifically, an online order is accepted at time $t$ if (and only if) the inventory remaining exceeds $\theta(t)$, where the optimal threshold function $\theta(\cdot)$ is non-decreasing in $t$ and independent of the starting inventory $n$. All offline orders are accepted while time remains. A continuous-time dynamic programming method is generally used to compute the optimal threshold function $\theta(\cdot)$; see \cite{Liang1999, Alishah2017} for details. Figure \ref{fig:ex1} illustrates numerically computed threshold functions for two settings with $\lambda_1 = 8$, $\lambda_2 = 2$, $p_1 = 25$, and $T = 5$.

Exact computation of $\theta(\cdot)$ requires recursion, repeated numerical integration, and floating-point calculations involving near-zero values. As a result, compounding numerical issues prevent reliable exact computation of thresholds as $T$ grows. We are therefore interested in characterizing the behavior of $\theta(t)$ as $t \to \infty$. \cite{Hodge2008} conjectures that $\theta(t)$ behaves linearly (that is, the function's step widths converge to a finite positive constant) as $t \to \infty$. Despite empirical evidence in support of the conjecture, such as Figure \ref{fig:ex1} of this work and Figure 4 of \cite{Alishah2017}, it remains unproven to our knowledge.

\begin{figure}[!h]
    \centering
    \includegraphics[width=0.8\textwidth]{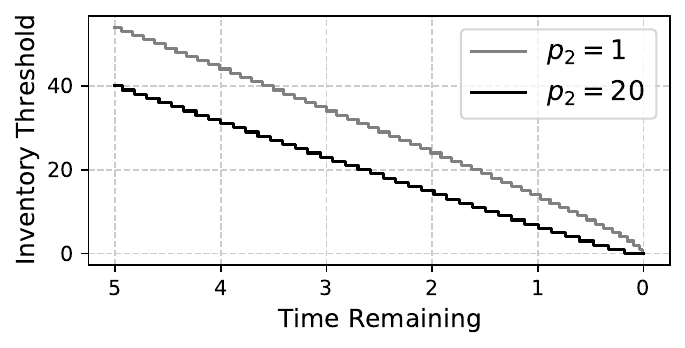}
    \vspace{-6pt}
    \caption{Example optimal threshold functions}
    \label{fig:ex1}
\end{figure}

In this work, we instead show that a range of simple linear inventory threshold policies are asymptotically optimal. Formally, for each $\beta \in (\lambda_1, \lambda_1 + \lambda_2)$, the linear threshold acceptance policy with slope $\beta$ (denoted $\beta$-\LT{}) accepts an incoming online customer at time $t$ if and only if the remaining inventory at that time is at least $\beta t$; all offline orders are accepted while time remains. To show that $\beta$-\LT{} policies are asymptotically optimal, we prove an even stronger result: for fixed $\lambda_1$ and $\lambda_2$, any appropriately chosen $\beta$-\LT{} policy achieves regret -- the expected absolute gap between a policy's revenue and that of the hindsight-optimal solution -- that is bounded above by a uniform constant that depends on $\beta$, $\lambda_1$, and $\lambda_2$ but not on $n$ and $T$. We then show that an extension of these policies also achieves uniformly bounded regret when more than two customer classes are present. Our policies and results are related to those of \cite{Arlotto2019} and \cite{Vera2021}. In particular, $\beta$-\LT{} policies are continuous-time generalizations of the policy proposed by \cite{Arlotto2019} for the multi-secretary problem. Although the approach used by \cite{Arlotto2019} could likely be modified to prove the results herein, our analysis (which directly uses Markov chain modeling and standard results from queueing theory) differs significantly from these prior works. 

\subsection{Related Work}

Our problem is a special case of the broader class of continuous-time, quantity-based network revenue management problems, commonly referred to simply as network revenue management (NRM). Specifically, our main problem can be viewed as a NRM problem with one resource (the product to be sold), unit demands, and two customer classes having stationary Poisson arrival rates. Heuristics for NRM typically rely on the solution to a deterministic linear programming (DLP) approximation of the stochastic problem. The DLP associated with our specific two-class problem is 
\begin{equation}
    \max_{z_1, z_2 \in [0, 1]} \lambda_1 T p_1 z_1 + \lambda_2 T p_2 z_2 \quad \text{s.t. } \lambda_1 T z_1 + \lambda_2 T z_2 \leq n,
    \label{baseDLP}
\end{equation}
where the variables $z_1$ and $z_2$ denote the proportion of offline and online orders accepted, respectively. One simple DLP-based heuristic accepts each incoming offline customer with probability $z_1$ and each incoming online customer with probability $z_2$ while inventory remains. This heuristic and many DLP-based heuristics for general NRM are known to achieve $\Theta(\sqrt{T})$ regret when $n$ and $T$ are scaled up proportionally.

\cite{Jasin2012} develop an 
NRM heuristic policy, \textit{Probabilistic Allocation Control} (\textsf{PAC}), that periodically re-solves the DLP  at equally spaced intervals. Although the authors show \textsf{PAC} to have $\bigO(1)$ regret if the DLP is non-degenerate, \cite{Bumpensanti2020} show that the \textsf{PAC} policy's regret is $\Theta(\sqrt{T})$ for certain degenerate DLPs. 
\cite{Bumpensanti2020} develop an improved NRM heuristic policy that achieves $\bigO(1)$ regret in the limit. Their approach, \textit{Infrequent Re-Solving with Thresholding}
(\IRT{}), requires $\bigO(\log \log T)$ re-optimizations of the underlying DLP at unequal intervals during a horizon of length $T$. 
For the special case of a single resource (i.e., product) considered in this work, we show that linear threshold policies also achieve $\bigO(1)$ regret without solving the DLP even once. 

\cite{Arlotto2019} construct a specific constant-regret \textit{Budget Ratio} (\BR{}) policy for a discrete-time multi-secretary problem. The continuous-time problem studied here can be viewed as a limiting case of their multi-secretary problem upon appropriate scaling. \cite{Vera2021} construct a specific constant-regret \textit{Fluid Bayes Selector} (\FBS{}) policy that uses dynamic thresholds for a class of dynamic optimization problems, of which NRM is a special case. In contrast to these works, our analysis explicitly assumes continuous-time arrivals, and our approach relies on characteristics of a particular discrete-time Markov chain (DTMC) model. Our approach shows that a range of $\beta$-\LT{} policies each achieve $\bigO(1)$ regret, suggesting that a range of policies similar to the authors' specifically constructed \BR{} and \FBS{} policies might also achieve constant regret. Additionally, as we discuss later, the asymptotic optimality of our $\beta$-\LT{} policies does not always require exact knowledge of the customer arrival rates. This is a continuous-time analog of the multi-secretary extension in which ``the ability distribution is...\ unknown to the decision maker'' noted by \cite{Arlotto2019} as a potential direction for future work. \revision{Finally, numerous other works study various extensions of the basic NRM problem. For example, \cite{Zhang2016} study a two-class problem with time-varying, correlated arrival rates and propose a heuristic that achieves regret no more than twice that of the optimal policy.}

\subsection{Definitions and Notation}

Let the arrival of offline customers over the interval $[T, 0]$ be denoted by the Poisson process $\{N_{1, T}(t), t \in [T, 0]\}$. Observe that, unlike a conventionally defined Poisson process, $N_{1, T}(t)$ is non-increasing in $t$ because $t$ decreases as time elapses. Similarly, let the arrival of online customers during $[T, 0]$ be denoted by the Poisson process $\{N_{2, T}(t), t \in [T, 0]\}$. For notational convenience, let $\mathcal{N}_T = \big( \{N_{1, T}(t)\}, \{N_{2, T}(t)\} \big)$ and $\lambda_{1, 2} = \lambda_1 + \lambda_2$.

Fundamentally, the retailer seeks to sell all $n$ units by the end of the horizon $t = 0$ while selling to as many offline customers as possible. That is, selling a unit to an online customer is preferable to leaving it unsold, but selling the unit to an offline customer is ideal. Therefore, with hindsight knowledge of $N_{1, T}(0)$ and $N_{2, T}(0)$, we can fully characterize the hindsight-optimal solution: sell $\min\{n,  N_{1, T}(0)\}$ units to offline customers, and $\min\big\{n - \min\{n,  N_{1, T}(0)\}, N_{2, T}(0) \big\}$ units to online customers. For given $n$ and $T$, the \textit{regret} $\rho_\text{P}(n, T)$ associated with a policy P is the expected difference between the revenue gained by the policy and the revenue gained by the hindsight-optimal solution. Regret must always be non-negative because no policy can outperform hindsight optimality. 


We use $f(x, y)$ to denote the probability of a Poisson random variable with mean $x$ taking the value $y$, and $f^+(x, y)$ for the probability of a Poisson random variable with mean $x$ taking a value of $y$ or greater. The function $h(w, x, y)$ denotes the probability density at $y$ of an Erlang random variable with shape $w$ and rate parameter $x$. We denote the positive integers by $\N$ and the non-negative integers by $\N_0$.

\section{Linear Threshold Policies}
\label{sec:two_class}


\cite{Hodge2008} proposed the heuristic $\beta$-\LT{} policy for two customer classes and empirically demonstrated its efficacy. We aim to show that, for any fixed $\beta \in (\lambda_1, \lambda_{1, 2})$, the $\beta$-\LT{} policy's regret $\rho_\beta(n, n / \alpha)$ is uniformly bounded above. 
This result is formalized in Theorem \ref{main}, the main result of this note. 

\begin{theorem}
    For each $\beta \in (\lambda_1, \lambda_{1, 2})$, there exists a finite positive constant $\hat{\rho}_\beta$ such that $\rho_\beta(n, n / \alpha) \leq \hat{\rho}_\beta$ for any $n \in \N$ and $\alpha > 0$.
    \label{main}
\end{theorem}

\noindent Theorem \ref{main} is implied by Corollary \ref{on_opt}, Proposition \ref{above_opt}, and Proposition \ref{below_opt}; these results correspond to the cases $\alpha = \beta$, $\alpha > \beta$, and $\alpha < \beta$, respectively. It should be noted that $\rho_\beta(\cdot, \cdot)$ and $\hat{\rho}_\beta$, as well as other functions and constants introduced later, depend on $\lambda_1$ and $\lambda_2$. However, we assume that these rate parameters are fixed beforehand unless stated otherwise, so $\lambda_1$ and $\lambda_2$ are omitted from subscripts and function arguments for brevity.

\subsection{Preliminaries}

We consider an identical system in which the inventory position can be negative. Specifically, once the inventory depletes, we continue to track the number of offline orders rejected due to the stockout. The \textit{final inventory position} is then the negative of the number of offline orders rejected after inventory depletion (if a stockout occurred) or the actual final positive inventory position otherwise. Let $\Delta_\beta(s, t)$ denote the absolute expected final inventory position under the $\beta$-\LT{} policy with $s$ inventory remaining and $t$ time remaining. 

For a given policy and realization of $\mathcal{N}_T$, the absolute deviation from the hindsight-optimal objective value is determined by two factors:
(i) the number of rejected customers of each type that should have been accepted while positive inventory remained, and
(ii) the number of accepted customers of each type that should have been rejected.
For a policy in which no offline customers are rejected until inventory depletion (e.g., a $\beta$-\LT{} policy) a final inventory position of $x > 0$ implies that achieving the hindsight-optimal objective value would have required accepting at most $x$ additional online customers. Similarly, a final position of $x < 0$ indicates that achieving the hindsight-optimal objective value would have required rejecting at most $y$ online customers and accepting at most $y$ offline customers instead, where $y \leq |x|$. A final position of zero indicates that hindsight optimality was achieved.
Therefore, under any $\beta$-\LT{} policy, $\rho_\beta(n, T) \leq p_1 \Delta_\beta(n, T)$ for any $n$ and $T$.

For a given $\beta$-\LT{} policy, starting time $T$, initial inventory $n$, and realization of $\mathcal{N}_T$, we refer to the evolution of the inventory position over inventory-time space as the $\beta$\textit{-induced sample path} or simply as the \textit{sample path} as in \cite{Hodge2008}. Formally, the position of the sample path at time $t$ can be represented by the function $R(t, \beta, n, \mathcal{N}_T)$. Figure \ref{fig:sample_path} illustrates an example sample path and $\beta$-line.
By our bound, we know that hindsight optimality is achieved if the sample path ends at zero inventory. Lemmas \ref{aboveHO} and \ref{belowHO} state that hindsight optimality is also achieved if the sample path never crosses the $\beta$-line during the horizon.

\begin{figure}[!h]
    \centering
    \includegraphics[width=\linewidth]{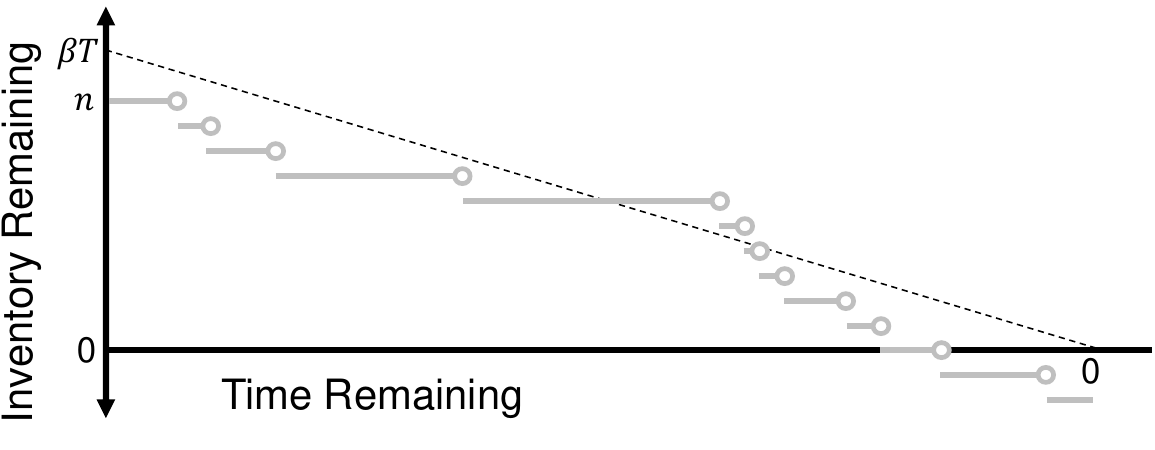}
    \caption{Example sample path (dark grey) relative to $\beta$-line (dashed black)}
    \label{fig:sample_path}
\end{figure}

\begin{lemma}
    If $n - \big(N_{1, T}(t) + N_{2, T}(t) \big) \geq \beta t$ for all $t \in [0, T]$, the $\beta$-\LT{} policy achieves hindsight optimality.\label{aboveHO}
\end{lemma}

\begin{proof}
    If $n - \big(N_{1, T}(t) + N_{2, T}(t) \big) \geq \beta t$ for all $t \in [0, T]$, the $\beta$-\LT{} policy and the hindsight-optimal policy both accept $N_{1, T}(0)$ offline customers and $N_{2, T}(0)$ online customers.
\end{proof}

\begin{lemma}
    If $n - N_{1, T}(t) < \beta t$ for all $t \in (0, T]$, the $\beta$-\LT{} policy achieves hindsight optimality.
    \label{belowHO}
\end{lemma}

\begin{proof}
    If $n - N_{1, T}(t) < \beta t$ for all $t \in (0, T]$, the $\beta$-\LT{} policy and the hindsight-optimal policy both accept $n$ offline customers and no online customers.
\end{proof}

\subsection{Starting on the $\beta$-Line}

Suppose first that the initial inventory lies on the $\beta$-line (i.e., $\alpha = \beta$); it will be evident later that this case is fundamental to analyzing the $\alpha > \beta$ and $\alpha < \beta$ cases. As in \cite{Hodge2008}, we define an \textit{excursion} to mean a segment of the sample path (in inventory-time space) that begins and ends on the $\beta$-line and does not intersect the $\beta$-line at any point in between. An excursion is almost surely comprised of two consecutive phases. In the first phase, the sample path begins on the $\beta$-line and stays strictly above the $\beta$-line. The second phase begins when an arrival causes the sample path to jump strictly below the $\beta$-line; the sample path then stays strictly below the $\beta$-line until it intersects the $\beta$-line from below, completing the excursion. 
If $\alpha = \beta$, the sample path's behavior during the horizon can be viewed as a series of independent, complete excursions, likely followed by a partial excursion truncated by the end of the horizon (unless, by chance, an excursion is completed exactly at $t = 0$). 



The original analysis of \cite{Hodge2008} focused primarily on deriving various statistics about complete excursions, such as the expected length of a complete excursion and the expected proportion of time spent above the $\beta$-line during a complete excursion. Because we seek to leverage the $\rho_\beta(n, T) \leq p_1 \Delta_\beta(n, T)$ bound, we are instead concerned with characterizing the expected inventory position upon truncation of the final partial excursion. To this end, we construct a DTMC model of the sample path by observing the system only at integer time points. Although our approach and intent are different from those of the original work, Hodge's observation that sample paths exhibit queue-like behavior is useful in analyzing the DTMC model.

\subsubsection{Inventory Position as a Discrete-Time Markov Chain}
\label{sec:MC_model}

Assume without loss of generality that $\alpha = \beta = 1$ and that all other parameters have been scaled correspondingly. For each $T \in \N_0$, we define the homogeneous DTMC $\left( X^T_t \right) \triangleq \left( X^T_T, X^T_{T - 1}, X^T_{T - 2}, \ldots \right)$ on $\Z$ as the sample path's position at time $t$ relative to the $\beta$-line, $R(t, \beta, T, \mathcal{N}_T) - \beta t$ = $R(t, 1, T, \mathcal{N}_T) - t$. As usual, inventory is allowed to be negative, and time counts backwards from $T$. Thus, $X^T_T$ is the chain's initial state, and $X^T_0$ is the inventory position at the end of the horizon. For mathematical purposes, we allow the DTMCs to evolve past $t = 0$ in the same manner so that $X^T_{-1}, X^T_{-2}, X^T_{-3}, \ldots$ are also well-defined.

For any chain $\left( X^T_t \right)$, the one-step transition probability from $i \in \Z$ to $j \in \Z$, denoted $q_{i, j}$, is the probability that the sample path is $j$ units above the $\beta$-line at time $t - 1$ given that the sample path is $i$ units above the $\beta$-line at time $t$. For $i \geq 1$ and $j \in \{1, 2, \ldots, i + 1\}$, we have $q_{i, j} = f(\lambda_{1, 2}, i - j + 1)$. For $i \leq -1$ and $j \leq i + 1$, we have \revision{$q_{i, j} = f(\lambda_1, i - j + 1)$}. For $i \geq 0$ and $j \leq 0$,
\begin{equation}
    q_{i, j} = \int_0^1 h\big(i + 1, \lambda_{1, 2}, u\big) f\big( (1 - u) \lambda_1, -j\big) \dv u.
    \label{qij_above_to_below}
\end{equation}
\revision{A derivation of this expression appears in Appendix \ref{proof:qij}.}
All other transition probabilities are zero. It is clear from the non-zero transition probabilities that each chain $\left( X^T_t \right)$ is irreducible and aperiodic. 

For all $i \in \Z$, let $\E[S_i]$ denote the expected number of steps for the DTMC to return to state $i$ after departing state $i$. \cite{Hodge2008} derives an expression for the expected duration of an excursion; in particular, the expression is always finite when $\lambda_1 < \beta < \lambda_{1, 2}$. By definition, this expected duration is simply $\E[S_0]$, and it therefore follows that $\E[S_0] < \infty$. The chain $\left( X^T_t \right)$ is therefore positive recurrent and ergodic, implying the existence of a unique stationary limiting distribution. Let $\bm{\pi} = (\pi_i)_{i \in \Z}$ denote the limiting distribution of $\left( X^T_t \right)$, where $\pi_i = 1/S_i > 0$ for all $i \in \Z$. Observe that $\bm{\pi}$ does not depend on $T$, because all of the chains $\big\{ \left( X^0_t \right), \left( X^1_t \right), \left( X^2_t \right), \ldots \big\}$ are time-homogeneous and have identical transition probabilities.

\revision{
\subsubsection{Restricted Chains and Stable Queues}

We aim to show that the sample path remains close to the $\beta$-line even after a very long period of time has elapsed. That is, we aim to show that the DTMC's expected absolute value remains finite at its limiting distribution. To do this, we analyze the original DTMC by decomposing it into two \textit{restricted chains}: one that tracks the value of the original DTMC \textit{only} when the sample path is at or above the $\beta$-line, and another that tracks the value of the original DTMC \textit{only} when the sample path is at or below the $\beta$-line.
}

\begin{figure}[!b]
    \centering
    \includegraphics[width=\linewidth]{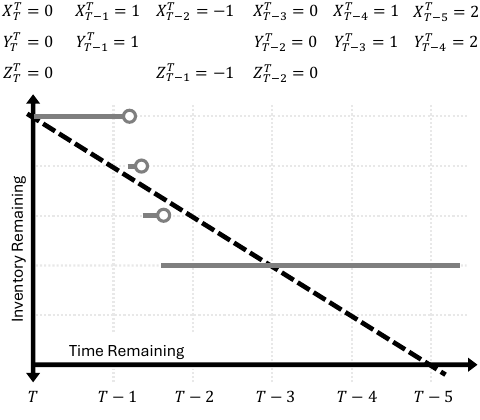}
    \caption{Example sample path (dark grey) relative to $\beta$-line (dashed black) with corresponding values of $\left( X^T_t \right)$, $\left(Y^T_\sigma\right)$, and $\left(Z^T_\tau\right)$}
    \label{fig:dtmc_illustration}
\end{figure}

For each $T \in \N_0$, let the DTMC $\left(Y^T_\sigma\right)$ denote the restriction of $\left( X^T_t \right)$ to $\N_0$. That is, for each $\sigma = T, T - 1, \ldots$, the value of $Y^T_\sigma$ is the value of $\left( X^T_t \right)$ at the $(T - \sigma + 1)$-th occurrence of $X^T_t \in \N_0$. 
\revision{Like $\left( X^T_t \right)$, the subscripts of $\left(Y^T_\sigma\right)$ count backwards from $T$. However, $\left(Y^T_\sigma\right)$ is only defined when $\left( X^T_t \right)$ is on or above the $\beta$-line. For example, in Figure \ref{fig:dtmc_illustration}, $\left(Y^T_{T - 4} \right) = 2$ as the 5th occurrence of the DTMC being on or above the $\beta$-line.}

Let $q^+_{i, j}$ denote the transition probabilities for this restricted chain, and let $\bm{\pi}^+$ denote the unique stationary limiting distribution of the restricted chain. For all $i \geq 0$ and $j \geq 1$, $q^+_{i,j} = q_{i,j}$. For all $i \geq 0$, we have $q^+_{i, 0} = \sum_{j = 0}^{-\infty} q_{i, j} = f^+(\lambda_{1, 2}, i + 1) = \revision{1 - \sum_{j = 0}^{\infty} q_{i, j}}$. \revision{In the context of the original DTMC, this is equal to the probability that the sample path is $i$ units above the $\beta$-line at time $t$ but crosses below the $\beta$-line before time $t - 1$. Then, the next time that the sample path is on or above the $\beta$-line must correspond to the next time that the original DTMC is in state 0 (hence the transition from $i$ to $0$ in the restricted chain).}

These $q^+_{i,j}$ are exactly the transition probabilities of the ergodic arrival-time Markov chain associated with a stable $D/M/1$ queue having arrival rate $\beta = 1$ and service rate $\lambda_{1, 2}$ (that is, the DTMC that tracks the number of jobs in the system just before the arrival of each new job). This queue's limiting arrival-time length distribution is known to have finite expected value (see e.g., \cite{Gross1998}, pp.\ 248--252). Therefore, $\sum_{i = 1}^\infty |i| \pi^+_i < \infty$.

For each $T \in \N_0$, let the DTMC $\left(Z^T_\tau\right)$ denote the restriction of $\left( X^T_t \right)$ to $\Z \setminus \N$. 
\revision{That is, for each $\tau = T, T - 1, \ldots$, the value of $Z^T_\tau$ is the value of $\left( X^T_t \right)$ at the $(T - \tau + 1)$-th occurrence of $X^T_t \in \Z \setminus \N$. 
Like $\left( X^T_t \right)$, the subscripts of $\left(Z^T_\tau\right)$ count backwards from $T$. However, $\left(Z^T_\tau\right)$ is only defined when $\left( X^T_t \right)$ is on or below the $\beta$-line. For example, in Figure \ref{fig:dtmc_illustration}, $\left(Z^T_{T - 1} \right) = -1$ as the 2nd occurrence of the DTMC being on or below the $\beta$-line.}

Let $q^-_{i, j}$ denote the transition probabilities for this restricted chain, and let $\bm{\pi}^-$ denote the unique stationary limiting distribution of the restricted chain. For all integers $i \leq -1$ and $j \leq 0$, $q^-_{i, j} = q_{i , j}$. Let the random variable $A$ denote the duration of time that an arbitrary excursion spends above the $\beta$-line, and let $g : [0, 1] \to \R^+$ denote the probability density function of $\lceil A \rceil - A$. Then, for all integers $j \leq 0$, we have 
\begin{equation}
    q^-_{0, j} = \int_0^1 g(u) f( u \lambda_1, j ) \dv u \ .
\end{equation}

For a stable $M/D/1$ queue, the associated ergodic departure-time DTMC tracks the number of jobs in the system just after the departure of each job. Let $\mathcal{Q}$ denote an $M/D/1$ queue having job arrival rate $\lambda_1$ and service rate $\beta > \lambda_1$. Let $\bm{\psi}$ denote the stationary distribution associated with the departure-time DTMC of $\mathcal{Q}$; it is known that $\sum_{i = 1}^\infty |i| \psi_i < \infty$ (see e.g., \cite{Serfozo2009}, pp.\ 316--318). Let $\mathcal{Q}'$ denote an $M/D/1$ queue having job arrival rate $\lambda_1$ and service rate $\beta$ with the additional feature that the first arriving job after each idle period is a partial job whose size is distributed according to $g(\cdot)$. Let $\bm{\omega}$ denote the stationary distribution associated with the departure-time DTMC of $\mathcal{Q}'$. For any particular sequence of job arrival times and any $k \in \N$, the number of jobs in the system upon the $k$-th job departure in $\mathcal{Q}'$ is no more than the number of jobs in the system upon the $k$-th job departure in $\mathcal{Q}$. Therefore, $\omega_i \leq \psi_i$ for all $i \in \N$, and thus $\sum_{i = 1}^\infty i \omega_i \leq \sum_{i = 1}^\infty i \psi_i < \infty$.


Observe that the departure-time DTMC of $\mathcal{Q}'$ has identical transition probabilities to the chains $\left(|Z^T_\tau|\right)$. It follows that $\omega_i = \pi^-_{-i}$ for all $i \in \N_0$, implying that $\sum_{i = -1}^{-\infty} |i| \pi^-_i < \infty$. \revision{For completeness, explicit derivations and comparisons of the queues' associated transition matrices are provided in Appendix \ref{queue_appendix}.}

Because $\left(Y^T_\sigma\right)$ and $\left(Z^T_\tau\right)$ are both restrictions of the corresponding $\left(X^T_t\right)$, the inequalities $\sum_{i = 1}^\infty |i| \pi^+_i < \infty$ and $\sum_{i = -1}^{-\infty} |i| \pi^-_i < \infty$ together imply $\sum_{i \in \Z} |i| \pi_i < \infty$. This allows us to prove that the expected absolute inventory position at $t = 0$ is bounded, independent of $T$, given that the initial inventory is on the $\beta$-line (i.e., the initial state is $X^T_T = 0$).

\begin{proposition}
    $\sup_{T \in \N} \big\{ \E\left[|X^T_0| \,\middle\vert\, X^T_T = 0 \right] \big\} \leq \frac{1}{\pi_0}\sum_{i \in \Z} |i| \pi_i < \infty$.\label{MC_bound}
\end{proposition}

\begin{proof}
    \revision{See Appendix \ref{proof:MC_bound}.}
\end{proof}

\noindent\revision{In other words, Proposition \ref{MC_bound} shows that the expected absolute value of the DTMC at $t = 0$ is at most $1/\pi_0$ times the expected absolute value of the stationary DTMC. This allows us to leverage the desirable property of the limiting distribution without requiring the DTMC to exhibit stationary behavior at $t = 0$.}

The Markov chain model was constructed to ensure $\big| \E\big[|X^n_0| \big| X^n_n = 0 \big] \big | = \Delta_\beta(n, n)$ for all $n \in \N$. As a result, Proposition \ref{MC_bound} directly implies Corollary \ref{on_opt}. 

\begin{corollary}
    \revision{Assume that all parameters have been re-scaled so that $\beta = 1$.}
    There exists $\eta_{\beta, 0} < \infty$ such that $\Delta_\beta(n, n) \leq \eta_{\beta, 0}$ for all $n \in \N$. Additionally, $\rho_\beta(n, n) \leq {\rho}^{=}_\beta$ for any $n \in \N$, where ${\rho}^{=}_\beta = p_1 \eta_{\beta, 0}$.
    \label{on_opt}
\end{corollary}

\subsection{Starting Above the $\beta$-Line}

We next study the case in which the sample path starts above the $\beta$-line. As a useful preliminary, we begin by showing that a similar uniform bound exists when the sample path begins no more than one unit below the $\beta$-line.

\begin{lemma}
    \revision{Assume that all parameters have been re-scaled so that $\beta = 1$.}
    There exists $\eta_{\beta, 1} < \infty$ such that $\left\{  \Delta_\beta(s - 1, s / \beta) \right\} \leq \eta_{\beta, 1}$ for all $s \in \N$.
    \label{unif_k_bound}
\end{lemma}

\begin{proof}
    The argument follows Section \ref{sec:MC_model} except that the DTMCs begin in state -1 at time $t = T$.
\end{proof}

\begin{lemma}
    \revision{Assume that all parameters have been re-scaled so that $\beta = 1$.}
    There exists $\hat{\eta}_\beta < \infty$ such that $\max_{s \in \R^+}\left\{  \Delta_\beta(\lfloor s\rfloor, s / \beta) \right\} \leq \hat{\eta}_\beta$.
    \label{unif_floor_bound}
\end{lemma}

\begin{proof}
    \revision{See Appendix \ref{proof:unif_floor_bound}.}
\end{proof}

If the initial inventory $n$ is above the $\beta$-line (i.e., $\alpha > \beta$), a sample path can evolve in one of two ways, both of which entail bounded regret. One possibility is that the sample path `jumps' below the $\beta$-line prior to the end of the horizon at least once. If it does so at some time $t > 0$, the inventory position is $\lfloor \beta t \rfloor$; thus, Lemma \ref{unif_floor_bound} implies that the expected final absolute inventory position is at most $\hat{\eta}_{\beta}$ in this case. The other possibility is that the sample path stays at or above the $\beta$-line for the entire horizon; in this case, the $\beta$-\LT{} policy gives the hindsight-optimal solution (Lemma \ref{aboveHO}). Proposition \ref{above_opt} formalizes this argument.

\begin{proposition}
    \revision{Assume that all parameters have been re-scaled so that $\beta = 1$.}
    There exists $\rho_\beta^+ < \infty$ such that $\rho_\beta(n, n / \alpha) \leq \rho_\beta^+$ for all $n \in \N$ and $\alpha > \beta$.\label{above_opt}
\end{proposition}

\begin{proof}
    Let $X(t)$ denote the sample path. Then,
    \begin{align}
        \rho_\beta(n, n / \alpha) & \leq 0 \cdot \Prob\big( X(t) \geq \beta t \enspace \forall t \in [0, T]\big) \nonumber \\
        & \quad + p_1 \hat{\eta}_\beta \cdot \Prob\big( \exists t \in [0, T] : X(t) < \beta t \big) \leq p_1 \hat{\eta}_\beta.
    \end{align}
    Thus, $\rho_\beta^+ = p_1 \hat{\eta}_\beta < \infty$ is the desired uniform bound.
\end{proof}

The proof also implies that, if the probability of crossing the $\beta$-line were to approach zero as $n \to \infty$, the regret would also tend to zero. This is indeed the case when $\alpha > \lambda_{1, 2}$.

\begin{corollary}
    For any $\alpha > \lambda_{1, 2}$ and $\beta \in (\lambda_1, \lambda_{1, 2})$, $\lim_{n \to \infty} \rho_\beta(n, n / \alpha) = 0$.
    \label{high_alpha_corollary}
\end{corollary}


\begin{proof}
    \revision{See Appendix \ref{proof:high_alpha_corollary}}
\end{proof}

\subsection{Starting Below the $\beta$-Line}

If the initial inventory $n$ is below the $\beta$-line (equivalently, $\alpha < \beta$), a sample path can evolve in one of two ways, both of which entail bounded regret. The first possibility is that the sample path intersects the $\beta$-line prior to the end of the horizon at least once; in this case, Corollary \ref{on_opt} implies that the expected final absolute inventory position is at most $\eta_{\beta, 0}$. The other possibility is that the sample path never intersects the $\beta$-line prior to the end of the horizon; in this case, the $\beta$-\LT{} policy gives the hindsight-optimal solution (Lemma \ref{belowHO}). The proof of Proposition \ref{below_opt} formalizes this argument.

\begin{proposition}
    \revision{Assume that all parameters have been re-scaled so that $\beta = 1$.}
    There exists $\rho_\beta^- < \infty$ such that $\rho_\beta(n, n / \alpha) \leq \rho_\beta^-$ for all $n \in \N$ and $\alpha < \beta$.\label{below_opt}
\end{proposition}

\begin{proof}
    Let $X(t)$ denote the $\beta$-induced sample path. Then,
    \begin{align}
        \rho_\beta(n, n / \alpha) & \leq 0 \cdot \Prob\big( X(t) < \beta t \enspace \forall t \in (0, T]\big) + \nonumber \\
        & \quad p_1 \eta_{\beta, 0} \cdot \Prob\big( \exists t \in (0, T] : X(t) = \beta t \big) \leq p_1 \eta_{\beta, 0} .
    \end{align}
    Thus, $\rho_\beta^- = p_1 \eta_{\beta, 0} = \rho_\beta^= < \infty$ is the desired uniform bound.
\end{proof}

\noindent Theorem \ref{main} is thus proved with $\hat{\rho}_\beta = \max\left\{ \rho_\beta^=, \rho_\beta^+, \rho_\beta^- \right\}$.

As before, the proof of Proposition \ref{below_opt} implies that regret would approach zero if the probability of intersecting the $\beta$-line were to approach zero. This is the case when $\alpha < \lambda_1$. The proof of Corollary \ref{low_alpha_corollary} mirrors that of Corollary \ref{high_alpha_corollary} (albeit using lower tail bounds) and is omitted for brevity.

\begin{corollary}
    For any $\alpha \in (0, \lambda_1)$ and $\beta \in (\lambda_1, \lambda_{1, 2})$, $\lim_{n \to \infty} \rho_\beta(n, n / \alpha) = 0$.
    \label{low_alpha_corollary}
\end{corollary}

\subsection{Discussion}

A benefit of $\beta$-\LT{} policies is that they may achieve asymptotic optimality even when $\lambda_1$ and $\lambda_2$ are not known with certainty. For example, if $\lambda_1 \in [18, 22]$ and $\lambda_2 \in [15, 20]$ but their exact values are unknown, then any $\beta$-\LT{} policy with $\beta \in (22, 33)$ is asymptotically optimal. Observe that the distributions of $\lambda_1$ and $\lambda_2$ were not required to deduce the viable range of $\beta$ --- only their supports. In contrast, other constant-regret policies require exact knowledge of $\lambda_1$ and $\lambda_2$. 

We conclude our discussion by highlighting a connection between the $\beta$-\LT{} policy and the DLP \eqref{baseDLP}. \cite{Hodge2008} derives the expected duration spent above and below the $\beta$-line during an excursion as follows.
\begin{theorem}[\cite{Hodge2008}, p.\ 69]
Let $\nu < 1$ denote the solution to $\nu e^{-\nu} = \lambda_{1, 2}\beta^{-1} e^{-\lambda_{1, 2}/\beta}$, and let $\kappa = \lambda_{1, 2} - \beta \nu$. The expected lengths of time above and below the $\beta$-line during an excursion are $\E[A] = 1/\kappa$ and $\E[B] = (\lambda_{1, 2} - \beta)/\kappa (\beta - \lambda_1)$,
respectively. 
\end{theorem}
\noindent The ratio between the expected time above the $\beta$-line (i.e., the expected time during which online orders are accepted) during an excursion and the expected total excursion duration is then
\begin{equation}
    \frac{\E[A]}{\E[A] + \E[B]} 
    = \frac{1}{\kappa} \left( \frac{1}{\kappa} + \frac{\lambda_{1, 2} - \beta}{\kappa (\beta - \lambda_1)} \right)^{-1}
    = \frac{\beta - \lambda_1}{\lambda_2}.
\end{equation}
This is exactly the optimal value of $z_2$ in \eqref{baseDLP} when the initial inventory lies on the $\beta$-line (i.e., $T = n/\beta$). Thus, assuming $T = n / \beta$ and a sufficiently long horizon, the optimal $z_2$ roughly coincides with the expected proportion of time during which online orders are accepted by the $\beta$-\LT{} policy. Although this result is not surprising, it is interesting to observe how we arrive at the optimal DLP proportion via an atypical route. 

The arguments underlying Corollaries \ref{high_alpha_corollary} and \ref{low_alpha_corollary} also imply that the $\beta$-\LT{} policy's online acceptance proportions approach one and zero when $\alpha > \lambda_{1, 2}$ and $\alpha < \lambda_1$, respectively; these are again the respective optimal $z_2$ proportions in each case. If we had a better understanding of the first crossing times of increasing linear boundaries by Poisson processes, particularly in regard to distributional behavior upon scaling, we would expect analogous results for $\alpha \in [\lambda_1, \beta) \cup (\beta, \lambda_{1, 2}]$. Unfortunately, we presently lack the necessary analytical results and/or bounds to formalize the argument for such $\alpha$. Because it is not a focus of this work, we defer to Chapter 3 of \cite{Zacks2017} for a summary of existing results on crossing time distributions for Poisson processes.


\section{Generalization to Multiple Customer Classes}

Suppose now that there are $K > 2$ distinct customer classes (denoted type-1, type-2, ..., type-$K$) with corresponding known arrival rates $\lambda_1, \lambda_2, \ldots, \lambda_K$ and selling prices $p_1 > p_2 > \cdots > p_K$. For each $j \in [K]$, the arrival of class-$j$ customers during $[T, 0]$ is given by the Poisson process $\{N_{j, T}(t), t \in [T, 0]\}$. In this $K$-class setting, we propose a generalization of the $\beta$-\LT{} policy. We define a $\boldsymbol{\beta}$-\LT{} policy with a vector $\bm{\beta} = (\beta_1, \beta_2, \ldots, \beta_{K - 1})$ where each $\beta_j$ satisfies \revision{$\sum_{i = 1}^{j} \lambda_i < \beta_j < \sum_{i = 1}^{j + 1} \lambda_i$}. Unlike in the two-class setting, the operation of the $\bm{\beta}$-\LT{} policy depends on the initial inventory-time ratio $\alpha = n/T$. 


Suppose that $\alpha \in [\beta_{k - 1}, \beta_k]$ for some \revision{$k \in \{2, \ldots, K - 1\}$}. The $\bm{\beta}$-\LT{} policy in this scenario proceeds as follows. We begin by accepting all customers of classes $\{1, 2, \ldots, k\}$ and no other customers until a stopping time $S$. We define $S$ as the earliest time at which the resulting sample path is not in the open cone between the $\beta_{k - 1}$-line and $\beta_k$-line. That is, 
\begin{equation}
    S = \max_{t \leq T} \Big\{t : n - \sum_{j = 1}^{k} N_{j, T}(t) \notin (\beta_{k - 1} t, \beta_k t) \Big\}.
\end{equation}

Observe that $S > 0$. 
After time $t = S$, the policy behaves in one of two ways. If $n - \sum_{j = 1}^{k} N_{j, T}(S) \leq \beta_{k - 1}S$, we operate analogously to the two-class policy with the $\beta_{k - 1}$-line guiding the acceptance decisions. That is, for all $t < S$, customers of classes $\{1, 2, \ldots, k - 1\}$ are always accepted, and customers of class $k$ are accepted if and only if the sample path is at or above the $\beta_{k - 1}$-line. Customers of classes $k + 1$ and above are never accepted. Because $\sum_{i = 1}^{k - 1} \lambda_i < \beta_{k - 1} < \sum_{i = 1}^{k} \lambda_i$, the analysis in Section \ref{sec:two_class} implies that the expected final inventory position in this case is bounded above by some constant $\zeta^+_{k - 1}$ that does not depend on the exact value of $n$ or $T$. In this case, a final inventory position of zero implies that hindsight optimality was achieved; the structure of the policy ensured that the optimal number of class-$k$ customers were accepted and none of the more profitable customers were rejected. A final inventory position of $x > 0$ implies that at most $x$ additional customers of classes $\{k, k + 1, \ldots, K\}$ should have been accepted to achieve hindsight optimality. A final inventory position of $y < 0$ implies that at most $y$ of the accepted customers should have been rejected instead to allow for at most $y$ additional sales to more profitable customers. Thus, the regret in this case is bounded by $p_1 \zeta^+_{k - 1}$.

The other case occurs if $n - \sum_{j = 1}^{k} N_{j, T}(S) \geq \beta_k S$. In this case, the $\beta_k$-line guides the acceptance decisions. That is, for all $t < S$, customers of classes $\{1, 2, \ldots, k\}$ are always accepted, and customers of class $k + 1$ are accepted if and only if the sample path is at or above the $\beta_{k}$-line. Customers of classes $k + 2$ and above are never accepted. Because $\sum_{i = 1}^{k} \lambda_i < \beta_{k} < \sum_{i = 1}^{k + 1} \lambda_i$, the analysis in Section \ref{sec:two_class} implies that the expected final inventory position in this case is bounded above by some constant $\zeta^-_{k}$ that does not depend on $n$ or $T$. Repeating the preceding analysis implies that the regret in this case is bounded above by $p_1 \zeta^-_{k}$.

It remains to consider the initial states for which $\alpha \notin [\beta_{k - 1}, \beta_k]$. When $\alpha < \beta_1$, the $\bm{\beta}$-\LT{} policy accepts only class-1 customers until the sample path intersects the $\beta_1$-line, at which point we simply follow the $\beta_1$-\LT{} policy from Section \ref{sec:two_class} for the remainder of the horizon. As before, the $\bm{\beta}$-\LT{} policy achieves hindsight optimality if the sample path stays below the $\beta_1$-line for the entire horizon. Thus, when $\alpha < \beta_1$, the regret is bounded above by some finite $p_1 \zeta_1^-$. When $\alpha > \beta_K$, the $\bm{\beta}$-\LT{} policy accepts all customers until the sample path crosses or intersects the $\beta_K$-line for the first time. After this time, we reject class-$K$ customers if and only if we are below the $\beta_K$-line; all other customers are always accepted. The $\bm{\beta}$-\LT{} policy achieves hindsight optimality if the sample path stays above the $\beta_K$-line for the entire horizon. Thus, when $\alpha > \beta_K$, the regret is bounded above by some finite $p_1 \zeta_K^+$. Generalizing over all $\alpha > 0$, the $\bm{\beta}$-\LT{} policy's regret is therefore bounded above by $\max \big\{ p_1 \zeta_1^-, p_1 \zeta_1^+, p_1 \zeta_2^-, p_1 \zeta_2^+, \ldots, p_1 \zeta_K^-, p_1 \zeta_K^+ \big\}$ for any $n$ and $T$.

\revision{
Recall that the $\bm{\beta}$-\LT{} policy never accepts customers of classes $k + 2$ and above when $\alpha \in [\beta_{k - 1}, \beta_k]$.
It may be possible to improve performance by allowing for the acceptance of some customers of classes $k + 2$ and above (as in the optimal policy of \cite{Liang1999}). However, our $\bm{\beta}$-\LT{} policy demonstrates that such acceptances are not required to achieve uniformly bounded regret.
}


\section{Numerical Results}

\subsection{Policy Comparison}

We consider a two-class problem with offline and online arrival rates of $\lambda_1 = \lambda_2 = 1$. Offline and online selling prices are set at $p_1 = 2$ and $p_1 = 1$, respectively, and $\alpha = n/T = 1.5$. We begin by simulating 10000 realizations for each $T \in \{50, 100, 500, 1000, 5000, 10000, 25000\}$, and we compare seven choices of $\beta$ between $\lambda_1$ and $\lambda_{1, 2}$. Table \ref{tab:basic_simulations} summarizes the results; the `Average HO' column gives the simulations' average hindsight-optimal objective.

\def\arraystretch{1.15}

\begin{table*}[!h]
\small
\centering
\caption{Simulated performance of $\beta$-\LT{} policies}
\begin{tabular}{rr|rrrrrrr|}
\cline{3-9}
\multicolumn{1}{l}{} & \multicolumn{1}{l|}{} & \multicolumn{7}{c|}{\textbf{Simulated} $\bm{\beta}$\textbf{-\LT{} Regret}} \\ \cline{2-9} 
\multicolumn{1}{l|}{} & \multicolumn{1}{c|}{\textbf{Average HO}} & \multicolumn{1}{c|}{$\bm{\beta = 1.05}$} & \multicolumn{1}{c|}{$\bm{\beta = 1.1}$} & \multicolumn{1}{c|}{$\bm{\beta = 1.25}$} & \multicolumn{1}{c|}{$\bm{\beta = 1.5}$} & \multicolumn{1}{c|}{$\bm{\beta = 1.75}$} & \multicolumn{1}{c|}{$\bm{\beta = 1.9}$} & \multicolumn{1}{c|}{$\bm{\beta = 1.95}$} \\ \hline
\multicolumn{1}{|r|}{$\bm{T}$ \textbf{= 50}} & 124.9353 & \multicolumn{1}{r|}{3.1432} & \multicolumn{1}{r|}{2.7147} & \multicolumn{1}{r|}{1.7761} & \multicolumn{1}{r|}{1.4060} & \multicolumn{1}{r|}{2.4515} & \multicolumn{1}{r|}{3.6997} & 4.1805 \\
\multicolumn{1}{|r|}{\textbf{100}} & 250.1043 & \multicolumn{1}{r|}{4.3772} & \multicolumn{1}{r|}{3.5370} & \multicolumn{1}{r|}{1.9691} & \multicolumn{1}{r|}{1.4428} & \multicolumn{1}{r|}{2.7614} & \multicolumn{1}{r|}{4.9310} & 5.8886 \\
\multicolumn{1}{|r|}{\textbf{500}} & 1250.2120 & \multicolumn{1}{r|}{7.5088} & \multicolumn{1}{r|}{4.6858} & \multicolumn{1}{r|}{1.9436} & \multicolumn{1}{r|}{1.4416} & \multicolumn{1}{r|}{2.9732} & \multicolumn{1}{r|}{7.8447} & 11.5840 \\
\multicolumn{1}{|r|}{\textbf{1000}} & 2500.0233 & \multicolumn{1}{r|}{8.7768} & \multicolumn{1}{r|}{4.8966} & \multicolumn{1}{r|}{1.9924} & \multicolumn{1}{r|}{1.4356} & \multicolumn{1}{r|}{3.0006} & \multicolumn{1}{r|}{8.5941} & 14.4756 \\
\multicolumn{1}{|r|}{\textbf{5000}} & 12500.1551 & \multicolumn{1}{r|}{9.7739} & \multicolumn{1}{r|}{4.8452} & \multicolumn{1}{r|}{1.9338} & \multicolumn{1}{r|}{1.4080} & \multicolumn{1}{r|}{2.9284} & \multicolumn{1}{r|}{8.7940} & 18.3506 \\
\multicolumn{1}{|r|}{\textbf{10000}} & 25000.2015 & \multicolumn{1}{r|}{9.9066} & \multicolumn{1}{r|}{4.8818} & \multicolumn{1}{r|}{1.9672} & \multicolumn{1}{r|}{1.4514} & \multicolumn{1}{r|}{2.9583} & \multicolumn{1}{r|}{8.6791} & 18.6464 \\
\multicolumn{1}{|r|}{\textbf{25000}} & 62501.7251 & \multicolumn{1}{r|}{9.8401} & \multicolumn{1}{r|}{4.8190} & \multicolumn{1}{r|}{1.9618} & \multicolumn{1}{r|}{1.4529} & \multicolumn{1}{r|}{2.9747} & \multicolumn{1}{r|}{8.7149} & 18.5961 \\ \hline
\end{tabular}
\label{tab:basic_simulations}
\end{table*}

As expected, regret tends to increase as $\beta \searrow \lambda_1 = 1$ and $\beta \nearrow \lambda_{1, 2} = 2$. Additionally, we need larger values of $T$ for the regret to stabilize when $\beta \searrow 1$ and $\beta \nearrow 2$. These phenomena are not symmetric over the interval $\beta \in (\lambda_1, \lambda_{1, 2})$, however. Figure \ref{fig:mid_regrets} plots the average regret for $T = 1000$ and $\beta \in \{1.01, 1.02, \ldots, 1.99\}$ across 10000 simulations. Of these choices, $\beta = 1.44$ performs the best with a simulated regret of 1.4001. All $\beta$ between 1.36 and 1.53 achieve regret less than 1.5 for $T = 1000$.

\begin{figure}[!h]
    \centering
    \includegraphics[width = \linewidth]{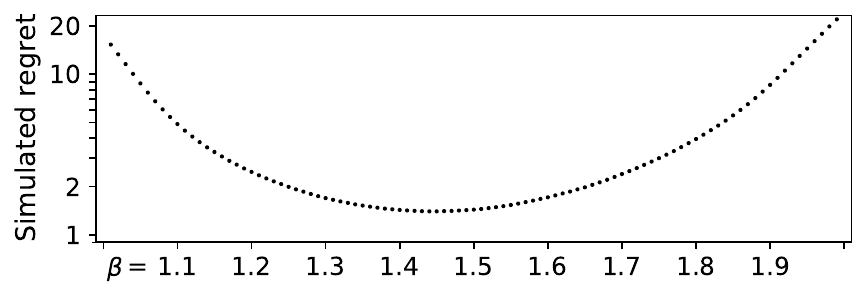}
    \vspace{-18pt}
    \caption{Simulated $\beta$-\LT{} regrets, $T = 1000$ (logarithmic vertical axis)}
    \label{fig:mid_regrets}
\end{figure}


We compare these results to Figure EC.1(e) of \cite{Bumpensanti2020} that summarizes policies' empirical performance for a problem with the same parameters. First, we observe that the \FBS{} results in that figure are competitive with our best $\beta$-\LT{} policies. Second, the \BR{} results in that figure are also competitive with our best $\beta$-\LT{} policies. This is as expected, because our $\beta$-\LT{} policy with $\beta = 1.5$ is essentially the \BR{} policy appropriately scaled to this continuous-time problem. Third, we observe that $\beta$-\LT{} policies with $\beta \approx 1.1$ and $\beta \approx 1.8$ entail regrets that are similar to those produced by \IRT{} for these parameters. All choices of $\beta$ between these values give even lower regrets, with a wide range of $\beta \in [1.25, 1.65]$ entailing regrets that are no more than half of those produced by \IRT{}.

Next, we compare $\beta$-\LT{} policies to the known optimal thresholding policy. Figure \ref{fig:comp_thresholds} illustrates the optimal policy for $t \leq 50$. To avoid computational issues for large $t$, we define an extrapolated optimal (\textsf{EO}) heuristic: we follow the exactly computed optimal policy for $t \leq 100$, and we linearly extrapolate the optimal threshold function for $t > 100$ based on the function's behavior at $t \approx 100$. The slope of the linear extrapolation is approximately 1.4420. Although this slope is close to the best-performing $\beta = 1.44$, this is merely a coincidence resulting from our choices of $p_1$ and $p_2$. For the values of $T$ in Table \ref{tab:basic_simulations}, \textsf{EO} achieves regrets of 1.3549, 1.4079, 1.4001, 1.3950, 1.3652, 1.4069, and 1.4091. These are all at most 0.01 lower than the corresponding regrets achieved by $\beta$-\LT{} with $\beta = 1.44$. 

\begin{figure}[!h]
    \centering
    \includegraphics[width = \linewidth]{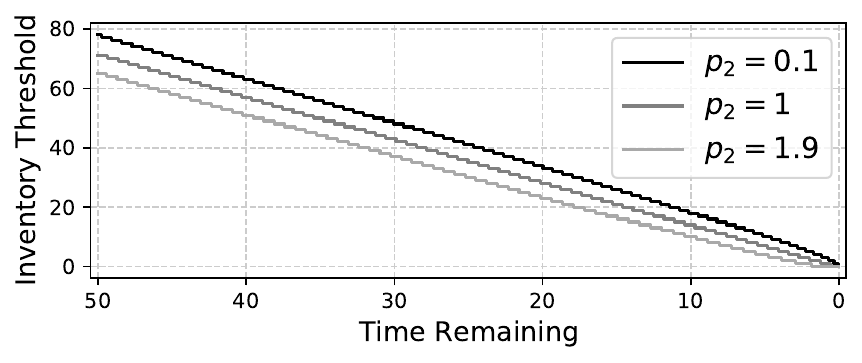}
    \vspace{-18pt}
    \caption{Optimal threshold functions, $p_2 \in \{0.1, 1, 1.9\}$}
    \label{fig:comp_thresholds}
\end{figure}

\subsection{Effects of Initial Inventory}

Our analysis indicates that sample paths are memoryless with respect to their initial conditions. Thus, one would expect that two values of $\alpha$ that are both associated with a high probability of crossing a particular $\beta$-line would entail similar regrets. We next conduct a brief numerical experiment to illustrate this idea. For each $\alpha \in \{1, 1.25, 1.75, 2\}$ and $\beta \in \{1.25, 1.75\}$, we simulate 10000 realizations for $T \in \{100, 1000, 10000\}$. Table \ref{tab:alpha_simulations} displays the simulated regrets.

\begin{table*}[!h]
\small
\centering
\caption{Simulated comparison of initial relative inventory levels}
\begin{tabular}{l|rrrrrrrr|}
\cline{2-9}
 & \multicolumn{8}{c|}{\textbf{Simulated} $\bm{\beta}$\textbf{-\LT{} Regret}} \\ \cline{2-9} 
 & \multicolumn{4}{c|}{$\bm{\beta = 1.25}$} & \multicolumn{4}{c|}{$\bm{\beta = 1.75}$} \\ \cline{2-9} 
 & \multicolumn{1}{r|}{$\bm{\alpha = 1}$} & \multicolumn{1}{r|}{$\bm{\alpha = 1.25}$} & \multicolumn{1}{c|}{$\bm{\alpha = 1.75}$} & \multicolumn{1}{c|}{$\bm{\alpha = 2}$} & \multicolumn{1}{c|}{$\bm{\alpha = 1}$} & \multicolumn{1}{c|}{$\bm{\alpha = 1.25}$} & \multicolumn{1}{c|}{$\bm{\alpha = 1.75}$} & \multicolumn{1}{c|}{$\bm{\alpha = 2}$} \\ \hline
\multicolumn{1}{|r|}{$\bm{T}$ \textbf{= 100}} & \multicolumn{1}{r|}{0.8196} & \multicolumn{1}{r|}{1.9506} & \multicolumn{1}{r|}{1.7675} & \multicolumn{1}{r|}{0.7479} & \multicolumn{1}{r|}{0.8213} & \multicolumn{1}{r|}{2.4067} & \multicolumn{1}{r|}{2.6624} & 1.0741 \\
\multicolumn{1}{|r|}{\textbf{1000}} & \multicolumn{1}{r|}{0.9424} & \multicolumn{1}{r|}{1.9924} & \multicolumn{1}{r|}{1.9921} & \multicolumn{1}{r|}{0.8804} & \multicolumn{1}{r|}{1.1384} & \multicolumn{1}{r|}{2.9997} & \multicolumn{1}{r|}{3.0006} & 1.3637 \\
\multicolumn{1}{|r|}{\textbf{10000}} & \multicolumn{1}{r|}{0.9678} & \multicolumn{1}{r|}{1.9672} & \multicolumn{1}{r|}{1.9672} & \multicolumn{1}{r|}{0.9509} & \multicolumn{1}{r|}{1.3522} & \multicolumn{1}{r|}{2.9583} & \multicolumn{1}{r|}{2.9583} & 1.4354 \\ \hline
\end{tabular}
\label{tab:alpha_simulations}
\end{table*}

As expected, any particular $\beta$-\LT{} policy performs similarly for each $\alpha \in \{1.25, 1.5, 1.75\}$. This effect is even more pronounced for larger values of $T$. When $T = 10000$, each $\alpha \in \{1.25, 1.5, 1.75\}$ results in exactly the same simulated regret of 1.9672 for $\beta = 1.25$. Also, when $T = 10000$, each $\alpha \in \{1.25, 1.5, 1.75\}$ results in exactly the same simulated regret of 1.9672 for $\beta = 1.25$. Each experiment uses the same set of 10000 random streams of customer arrivals. Thus, for each individual trial, these results indicate that the three sample paths beginning at each $\alpha \in \{1.25, 1.5, 1.75\}$ typically coincide at some point and move together for the remainder of the selling horizon (i.e., the sample paths couple) for large values of $T$.

For any $\alpha \in (\lambda_1, \lambda_1 + \lambda_2)$ and any $\beta \in (\lambda_1, \lambda_1 + \lambda_2)$, the probability of crossing the $\beta$-line during the horizon approaches one as $T$ grows. This is not the case when $\alpha = \lambda_1$ or $\alpha = \lambda_1 + \lambda_2$. Recall that a sample path achieves hindsight optimality if it never crosses the $\beta$-line during the horizon. Thus, the simulated regrets for $\alpha \in \{1, 2\}$ are much lower than those for $\alpha \in \{1.25, 1.5, 1.75\}$, even for large values of $T$.

\subsection{Effects of Selling Prices}

The risk associated with accepting an online customer decreases as $p_2/p_1$ grows. Hence, we expect lower ideal values of $\beta$ for higher $p_2/p_1$ and vice versa. Empirical evidence supports this idea. Returning to the original setting of $\alpha = 1.5$, Figure \ref{fig:hilo_regrets} shows the regret across 10000 simulations for $T = 1000$, $p_1 = 2$, and $p_2 \in \{0.1, 1.9\}$. For $p_2 = 0.1$, the lowest regret of 0.7041 is achieved at $\beta = 1.78$. For $p_2 = 1.9$, the lowest regret of 0.5127 is achieved at $\beta = 1.17$. These two $\beta$-\LT{} policies are again competitive with their corresponding \textsf{EO} policies, with each achieving regret no higher than 0.3 above \textsf{EO} for all of the values of $T$ in Table \ref{tab:basic_simulations}. These numerical results for $p_2 \in \{0.1, 1, 1.9\}$ suggest that the expected performance of an appropriately chosen $\beta$-\LT{} policy may be nearly indistinguishable from that of the optimal threshold policy.

\begin{figure}[!h]
    \centering
    \includegraphics[width = \linewidth]{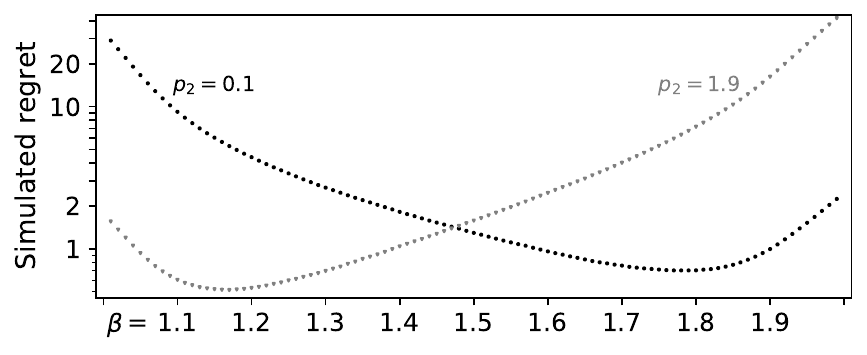}
    \vspace{-16pt}
    \caption{Simulated $\beta$-\LT{} regrets with varying offline selling price, $T = 1000$ (logarithmic vertical axis)}
    \label{fig:hilo_regrets}
\end{figure}

Figure \ref{fig:comp_thresholds} displays the exactly computed threshold functions for $t \leq 50$ and $p_2 \in \{0.1, 1, 1.9\}$. Although the optimal threshold functions' behavior near the end of the horizon is clearly influenced by $p_2$, the three functions' slopes seem to coincide as $t \to \infty$! For these values of $p_2$, the optimal threshold functions' step sizes all appear to converge to approximately 0.69 (corresponding to a slope of approximately 1.44) as $t$ grows. 



Recall the original conjecture of \cite{Hodge2008} in Section \ref{sec:intro}. We conclude this work by conjecturing an even stronger statement: as $t \to \infty$, the step widths of the optimal threshold function $\theta(t)$ converge to a finite positive constant that depends on $\lambda_1$ and $\lambda_2$ but \underline{not} on $p_1$ or $p_2$ (assuming $p_1 > p_2 > 0$).

\section*{Acknowledgements}
The authors acknowledge financial support from The Home Depot. Dipayan Banerjee’s work was also supported by the U.S.\ National Science Foundation (DGE-2039655). The authors thank He Wang for helpful comments. The authors have no further interests to declare.

{\bibliography{_refs}}



\newpage

\appendix

\section{\revision{Omitted Proofs and Technical Details}}
\label{proofs}

\revision{
\subsection{Derivation of Equation \eqref{qij_above_to_below}}
\label{proof:qij}

We are interested in the transition probability $q_{i, j}$ for the case in which $i \geq 0$ and $j \leq 0$. In this case, the sample path is on or above the $\beta$-line during the earlier part of the time interval $[t, t - 1]$, then a customer arrives which moves the sample path below the $\beta$-line, then the sample path is below the $\beta$-line for the remainder of the interval. A total of $i + 1 + |j|$ customers must be accepted during the interval: $i + 1$ customers of any type must arrive while the sample path is above the $\beta$-line, then $|j|$ offline customers must arrive while the sample path is below the $\beta$-line.

Therefore, two independent events must occur in succession. First, exactly $i + 1$ customer arrivals (of any type) must occur during the interval prior to $t - 1$, which brings the sample path just below the $\beta$-line. The time at which the $(i + 1)$-th of these customers arrive is denoted as $v \in (t, t - 1)$. By definition, the probability density of $v$ is Erlang: $h(i + 1, \lambda_{1, 2}, t - v)$. 
Upon crossing the $\beta$-line at time $v$, the second necessary event is the arrival of exactly $|j| = -j$ offline customers during the time interval $(v, t - 1]$. The probability of this event occurring, conditional on the sample path crossing the $\beta$-line at time $v$, is $f\big( (v - (t - 1)) \lambda_1, -j\big)$. Figure \ref{fig:qij_illustration} illustrates these two events.

\begin{figure}[!h]
    \centering
    \vspace{-8pt}
    \includegraphics[width=0.85\linewidth]{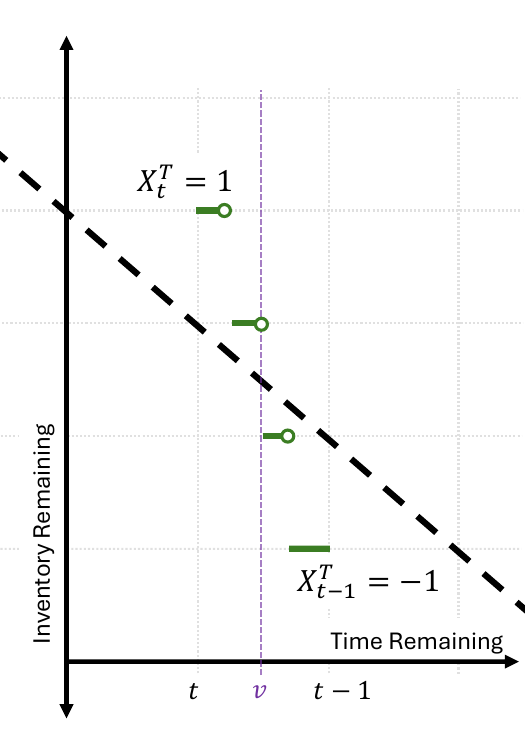} \vspace{-18pt}
    \caption{\revision{Example sample path (in green) relative to $\beta$-line (in black, dashed) for $X^T_t = 1$ and $X^T_{t - 1} = -1$; the crossing time $v$ is marked in purple}}
    \label{fig:qij_illustration}
\end{figure}

The probability density of both of these events occurring, with the crossing of the $\beta$-line taking place at time $v$, is therefore
\begin{equation}
    h(i + 1, \lambda_{1, 2}, t - v) f\big( (v - (t - 1)) \lambda_1, -j\big)
\end{equation}
by the independence of customer arrivals. To calculate $q_{i, j}$, we must then integrate over all possible values of $v$ between $t$ and $t - 1$:
\begin{equation}
    q_{i, j} = \int_{t}^{t - 1} h(i + 1, \lambda_{1, 2}, t - v) f\big( (v - (t - 1)) \lambda_1, -j\big) \dv v .
\end{equation}
The substitution $u = t - v$ produces the expression
\begin{equation*}
    q_{i, j} = \int_0^1 h\big(i + 1, \lambda_{1, 2}, u\big) f\big( (1 - u) \lambda_1, -j\big) \dv u
\end{equation*}
as desired.
}


\revision{
\subsection{$D/M/1$ and $M/D/1$ Transition Matrices}
\label{queue_appendix}

The one-step transition matrix of $\left(Y^T_\sigma\right)$ can be written as
\begin{equation}
\begin{bmatrix}
1 - f(\lambda_{1, 2}, 0) & f(\lambda_{1, 2}, 0) & 0 & 0 & 0 & \cdots \\
1 - \sum_{\ell = 0}^1 f(\lambda_{1, 2}, \ell)& f(\lambda_{1, 2}, 0) & f(\lambda_{1, 2}, 1) & 0 & 0 & \cdots \\
1 - \sum_{\ell = 0}^2 f(\lambda_{1, 2}, \ell) & f(\lambda_{1, 2}, 0) & f(\lambda_{1, 2}, 1) & f(\lambda_{1, 2}, 2) & 0 & \cdots \\
\vdots & \vdots & \vdots & \vdots & \vdots & \ddots
\end{bmatrix} .
\end{equation}
The one-step transition matrix of a $G/M/1$ queue is given in equation (5.54) of \cite{Gross1998} as 
\begin{equation}
\begin{bmatrix}
1 - b_0 & b_0 & 0 & 0 & 0 & \cdots \\
1 - \sum_{\ell = 0}^1 b_\ell& b_0 & b_1 & 0 & 0 & \cdots \\
1 - \sum_{\ell = 0}^2 b_\ell & b_0 & b_1 & b_2 & 0 & \cdots \\
\vdots & \vdots & \vdots & \vdots & \vdots & \ddots
\end{bmatrix} .
\end{equation}
Setting 
\begin{equation*}
b_\ell = \frac{e^{-\lambda_{1, 2}} \lambda_{1, 2}^\ell}{\ell!} ,     
\end{equation*}
the $D/M/1$ special case of equation (5.53) in \cite{Gross1998}, shows the equivalence of the two transition matrices.


Theorem 119 of \cite{Serfozo2009} gives the one-step transition matrix of the departure-time DTMC of an $M/D/1$ queue $\mathcal{Q}$ (with arrival rate $\lambda_1$ and departure rate $\beta = 1$) as 
\begin{equation}
\begin{bmatrix}
f(\lambda_{1}, 0) & f(\lambda_{1}, 1) & f(\lambda_{1}, 2) & \cdots \\
f(\lambda_{1}, 0) & f(\lambda_{1}, 1) & f(\lambda_{1}, 2) & \cdots \\
0   & f(\lambda_{1}, 0) & f(\lambda_{1}, 1) & \cdots \\
0   & 0   & f(\lambda_{1}, 0) & \cdots \\
\vdots & \vdots & \vdots & \ddots
\end{bmatrix} ,
\label{Q_dtmc}
\end{equation}
This DTMC is known to have a stationary limiting distribution and a finite expected stationary value for $\lambda_1 < 1$. The departure-time DTMC of an $M/D/1$ queue $\mathcal{Q}'$ having the additional feature that the first arriving job after each idle period is a partial job whose size is distributed according to $g(\cdot)$ is then
\begin{equation}
\begin{bmatrix}
\int_0^1 g(u) f( u \lambda_1, 0 ) \dv u & \int_0^1 g(u) f( u \lambda_1, 1 ) \dv u & \int_0^1 g(u) f( u \lambda_1, 2 ) \dv u & \cdots \\
f(\lambda_{1}, 0) & f(\lambda_{1}, 1) & f(\lambda_{1}, 2) & \cdots \\
0   & f(\lambda_{1}, 0) & f(\lambda_{1}, 1) & \cdots \\
0   & 0   & f(\lambda_{1}, 0) & \cdots \\
\vdots & \vdots & \vdots & \ddots
\end{bmatrix} .
\label{Q'_dtmc}
\end{equation}
Observe that these are exactly the one-step transition probabilities of $\left(|Z^T_\tau|\right)$. For all $\ell \in \N$, it holds that 
\begin{subequations}
    \begin{align}
        \int_0^1 g(u) f( u \lambda_1, \ell ) \dv u
        & < \int_0^1 g(u) f( \lambda_1, \ell ) \dv u \\
        &  = f( \lambda_1, \ell ) \int_0^1 g(u) \dv u \\
        & = f( \lambda_1, \ell ),
    \end{align}
\end{subequations}
where $f( u \lambda_1, \ell ) < f( \lambda_1, \ell )$ is a consequence of \eqref{partial_dv} discussed later. Thus, there exist $\varepsilon_1, \varepsilon_2, \ldots > 0$ such that we can rewrite \eqref{Q'_dtmc} as 
\begin{equation*}
\begin{bmatrix}
f(\lambda_{1}, 0) + \sum_{\ell = 1}^\infty \varepsilon_\ell & f(\lambda_{1}, 1) - \varepsilon_1 & f(\lambda_{1}, 2) - \varepsilon_2 & \cdots \\
f(\lambda_{1}, 0) & f(\lambda_{1}, 1) & f(\lambda_{1}, 2) & \cdots \\
0   & f(\lambda_{1}, 0) & f(\lambda_{1}, 1) & \cdots \\
0   & 0   & f(\lambda_{1}, 0) & \cdots \\
\vdots & \vdots & \vdots & \ddots
\end{bmatrix} .
\end{equation*}
This is exactly the transition matrix \eqref{Q_dtmc} except that the $0 \to \ell$ transition probabilities have been reduced for all $\ell \in \N$, with a corresponding increase in the $0 \to 0$ transition probability. Then, if $\bm{\psi}$ is the stationary distribution associated with \eqref{Q_dtmc} and $\bm{\omega}$ is the stationary distribution associated with \eqref{Q'_dtmc}, it follows that $\sum_{i = 1}^\infty i \omega_i \leq \sum_{i = 1}^\infty i \psi_i < \infty$.

}

\subsection{Proof of Proposition \ref{MC_bound}}
\label{proof:MC_bound}

\begin{proof}
    For any $T \in \N$, we have $\E\left[|X^T_0| \,\middle\vert\, X^T_T = 0 \right] = \E\left[|X^0_{-T}| \,\middle\vert\, X^0_0 = 0 \right]$ by definition. Hence, it suffices to show $\sup_{t \in \Z \setminus \N} \big\{\E\left[|X^0_t| \,\middle\vert\, X^0_0 = 0 \right]\big\} < \infty$.

    Let $\bm{\pi}$ denote the stationary distribution, which we know to satisfy $\sum_{i \in \Z} |i| \pi_i < \infty$. Then, for any $t \in \Z \setminus \N$,
    \begin{subequations}
    \begin{align}
     \E\left[|X^0_t| \,\middle\vert\, X^0_0 \sim \bm{\pi} \right] 
    & = \sum_{i \in \Z} |i| \Prob(X^0_t = i \mid X^0_0 \sim \bm{\pi})  \\ 
    & = \sum_{i \in \Z} |i| \pi_i.
    \end{align}
    \end{subequations}
    Additionally, for any $t \in \Z \setminus \N$,
    \begin{subequations}
        \begin{align}
            \E\left[|X^0_t| \,\middle\vert\, X^0_0 \sim \bm{\pi} \right]
        & = \sum_{i \in \Z} \E\left[|X^0_t| \,\middle\vert\, X^0_0 = i \right] \Prob(X^0_0 = i) \\        
        & = \sum_{i \in \Z} \E\left[|X^0_t| \,\middle\vert\, X^0_0 = i \right] \pi_i \\
        & \revision{= \E\left[|X^0_t| \,\middle\vert\, X^0_0 = 0 \right] \pi_0 \ +} \\
        & \revision{\qquad \qquad 
        \sum_{i \in \Z \setminus \{0\}} \E\left[|X^0_t| \,\middle\vert\, X^0_0 = i \right] \pi_i} \nonumber \\
        & \geq \E\left[|X^0_t| \,\middle\vert\, X^0_0 = 0 \right] \pi_0, \
        \end{align}
    \end{subequations}
    \revision{where the final step is implied by the fact that
    \begin{equation}
        \sum_{i \in \Z \setminus \{0\}} \E\left[|X^0_t| \,\middle\vert\, X^0_0 = i \right] \pi_i \geq 0.
    \end{equation}}
    
    \noindent Therefore, $\E\left[|X^0_t| \,\middle\vert\, X^0_0 = 0 \right] \leq \frac{1}{\pi_0}\sum_{i \in \Z} |i| \pi_i < \infty$. 
\end{proof}

\subsection{Proof of Lemma \ref{unif_floor_bound}}
\label{proof:unif_floor_bound}

\begin{proof}
    Without loss of generality, assume that $\beta = 1$ and all that other parameters have been scaled appropriately. Let $s > 0$, and assume that $s \notin \N$ to avoid the case from Corollary \ref{on_opt}. 

    \revision{
    To begin, we rewrite the expectation $\Delta_\beta \big(\lfloor s\rfloor, s \big)$ using the tower property, conditioning on the number of offline arrivals during the time interval $\big(s, \lfloor s \rfloor\big)$. 
    The expression $f\big( (s - \lfloor s \rfloor) \lambda_1, k\big)$ represents the probability of $k$ offline arrivals in the time interval $\big(s, \lfloor s \rfloor\big)$.
    The expression $\Delta_\beta \big(\lfloor s\rfloor - k, \lfloor s\rfloor \big)$ represents the expected absolute final inventory position given that $k$ offline customers arrived in the time interval $\big(s, \lfloor s \rfloor\big)$. Then, we have
    \begin{subequations}
    \begin{align}
        \Delta_\beta \big(\lfloor s\rfloor, s \big)
        & = \sum_{k = 0}^{\infty} f\big( (s - \lfloor s \rfloor) \lambda_1, k\big) 
        \Delta_\beta \big(\lfloor s\rfloor - k, \lfloor s\rfloor \big) \label{floor_lemma_step0}\\
        & = f\big( (s - \lfloor s \rfloor) \lambda_1, 0\big) \Delta_\beta(\lfloor s\rfloor, \lfloor s\rfloor ) \ + \label{floor_lemma_step1} \\   
        & \qquad \sum_{k = 1}^{\infty} f\big( (s - \lfloor s \rfloor) \lambda_1, k\big) 
        \Delta_\beta \big(\lfloor s\rfloor - k, \lfloor s\rfloor \big) . \nonumber 
    \end{align}
    \end{subequations}
    
    Next, for any $k \in \N$ and $x \in (0, 1)$,
    \begin{equation}
        \frac{\partial}{\partial x} f(x, k) =
        \frac{\partial}{\partial x} \left( \frac{x^k e^{-x}}{k!} \right)
        = \frac{x^{k - 1} e^{-x}}{k!} (k - x) > 0 .
        \label{partial_dv}
    \end{equation}
    Because $\lambda_1 < \beta = 1$ and $s - \lfloor s \rfloor < 1$, we therefore have 
    \begin{equation}
        f\big( (s - \lfloor s \rfloor) \lambda_1, k\big) \leq f(\lambda_1, k) = f\big( (\lceil s \rceil - \lfloor s \rfloor) \lambda_1, k\big)
    \end{equation}
    for all $k \in \N$. Applying this inequality to \eqref{floor_lemma_step1} gives
    \begin{align}
        \Delta_\beta \big(\lfloor s\rfloor, s \big)
        & \leq f\big( (s - \lfloor s \rfloor) \lambda_1, 0\big) \Delta_\beta(\lfloor s\rfloor, \lfloor s\rfloor ) \ + \\   
        & \qquad \sum_{k = 1}^{\infty} f\big( (\lceil s \rceil - \lfloor s \rfloor) \lambda_1, k\big) \Delta_\beta(\lfloor s\rfloor - k, \lfloor s\rfloor ) .\nonumber 
    \end{align}
    This further implies
    \begin{align}
        \Delta_\beta \big(\lfloor s\rfloor, s \big)
        & \leq \Delta_\beta(\lfloor s\rfloor, \lfloor s\rfloor ) \ +  \label{floor_lemma_step2}\\   
        & \qquad \sum_{k = 0}^{\infty} f\big( (\lceil s \rceil - \lfloor s \rfloor) \lambda_1, k\big) \Delta_\beta(\lfloor s\rfloor - k, \lfloor s\rfloor )   \nonumber 
    \end{align}
    because $f\big( (s - \lfloor s \rfloor) \lambda_1, 0\big) \leq 1$ and $f\big( (\lceil s \rceil - \lfloor s \rfloor) \lambda_1, 0\big) \Delta_\beta(\lfloor s\rfloor - 0, \lfloor s\rfloor ) \geq 0$. 

    Finally, the same conditional expectation argument of \eqref{floor_lemma_step0} gives 
    \begin{equation}
        \Delta_\beta(\lfloor s \rfloor, \lceil s \rceil ) = \sum_{k = 0}^{\infty} f\big( (\lceil s \rceil - \lfloor s \rfloor) \lambda_1, k\big) \Delta_\beta(\lfloor s\rfloor - k, \lfloor s\rfloor ) \label{floor_lemma_step3}
    \end{equation}
    as an expression for the expected absolute final inventory position with $\lfloor s \rfloor$ inventory remaining and $\lceil s \rceil$ time remaining. Combining \eqref{floor_lemma_step2} and \eqref{floor_lemma_step3} gives
    \begin{equation}
        \Delta_\beta \big(\lfloor s\rfloor, s \big) \leq \Delta_\beta(\lfloor s\rfloor, \lfloor s\rfloor) + \Delta_\beta(\lfloor s \rfloor, \lceil s \rceil ).
    \end{equation}}
    By Corollary \ref{on_opt} and Lemma \ref{unif_k_bound}, $\hat{\eta}_\beta = \eta_{\beta, 0} + \eta_{\beta, 1} < \infty$.
    \end{proof}

    

\subsection{Proof of Corollary \ref{high_alpha_corollary}}
\label{proof:high_alpha_corollary}

\begin{proof}
    It remains to be shown that the probability of crossing the $\beta$-line at least once during the horizon approaches zero as $n \to \infty$ and $T = n / \alpha \to \infty$; denote this probability $\xi_\beta(n, n/\alpha)$. Because $\lambda_{1, 2} > \beta$, $\xi_\beta(n, n/\alpha)$ is bounded above by the probability of the sample path crossing the $\lambda_{1, 2}$-line (defined analogously to the $\beta$-line) at least once during the horizon.
    
    For all $n \in \N_0$, let $n \delta = n\big( 1 - \lambda_{1, 2}/\alpha \big) = n - \lambda_{1, 2}n/\alpha > 0$ denote the expected final inventory position if all customers were to be accepted (allowing for negative inventory positions as usual). For all $t \in [0, T]$, observe that $\lambda_{1, 2} t = n - \E\big[ N_{1, n/\alpha}(t) + N_{2, n/\alpha}(t) \big] - n\delta$ for all $t \in [0, T]$. Without loss of generality, assume that $\alpha = 1$ (so that $n = T$) and all other parameters have been scaled appropriately. Then, for any $n \in \N$,
    \begin{subequations}
    \begin{align}
        & \xi_{\beta}(n, n/\alpha) \leq \Prob\left( \exists t \in [0, n] : n - \big(N_{1, n/\alpha}(t) + N_{2, n/\alpha}(t) \big) < \lambda_{1, 2} t \right) \notag \\
        & \ = \Prob\big( \exists t \in \{0, 1, \ldots, n - 1\} : n - \big(N_{1, n/\alpha}(t) + N_{2, n/\alpha}(t) \big) \leq \lambda_{1, 2} t \big) \notag \\
        & \ = \Prob\big( \exists t \in \{0, 1, \ldots, n - 1\} : N_{1, n/\alpha}(t) + N_{2, n/\alpha}(t) \geq n - \lambda_{1, 2} t \big) \notag \\
        & \ = \Prob\big( \exists t \in \{0, 1, \ldots, n - 1\} : N_{1, n/\alpha}(t) + \notag \\
        & \qquad \qquad N_{2, n/\alpha}(t) \geq \E\big[ N_{1, n/\alpha}(t) + N_{2, n/\alpha}(t) \big] + n\delta \big) \notag \\
        & \ \leq \sum_{t = 0}^{n - 1} \Prob\left( N_{1, n/\alpha}(t) + N_{2, n/\alpha}(t) \geq \E\big[ N_{1, n/\alpha}(t) + N_{2, n/\alpha}(t) \big] + n\delta \right) \notag \\
        & \ \leq \sum_{t = 0}^{n - 1} \exp \left( \frac{-\delta^2 n^2}{2 \left( \E\big[ N_{1, n/\alpha}(t) + N_{2, n/\alpha}(t) \big] + n\delta\right)}\right) 
        \label{eq:PoisTail}\\ 
        & \ \leq n \exp \left( -\delta^2 n^2 / \left[ 2 \left( \lambda_{1, 2}n + n\delta\right) \right] \right)\\
        & \ \leq \frac{n}{\exp\big( \delta^2 n/(2\lambda_1 + 2\lambda_2 + \revision{2\delta}) \big)},\label{eq:last_step}
    \end{align}
    \end{subequations}
where \eqref{eq:PoisTail} uses the Poisson upper tail bound in Theorem A.8 of \cite{Canonne2022}, and  \eqref{eq:last_step} uses the fact that $\E\big[ N_{1, n/\alpha}(t) + N_{2, n/\alpha}(t)\big] \leq \lambda_{1, 2}n$ for all $t \in [0, n]$. The final expression goes to zero as $n \to \infty$, giving the desired result and an upper bound on the rate of convergence.    
\end{proof}

\vfill

\end{document}